\documentclass{amsart}
\usepackage{amsmath, amssymb, amsthm, amscd, amsfonts, eucal, hyperref}
\usepackage{enumerate}
\usepackage[T5, T1]{fontenc}

\newtheorem{theorem}{Theorem}[section]
\newtheorem{proposition}[theorem]{Proposition}
\newtheorem{lemma}[theorem]{Lemma}
\newtheorem {corollary}[theorem]{Corollary}
\theoremstyle {definition}
\newtheorem {definition}[theorem]{Definition}
\newtheorem {example}[theorem]{Example}
\newtheorem {conjecture}[theorem]{Conjecture}
\theoremstyle {remark}

\def\Spec{\operatorname{Spec}}

\def\Proj{\operatorname{Proj}}
\def\Ann{\operatorname{Ann}}
\def\Ass{\operatorname{Ass}}
\def\Supp{\operatorname{Supp}}
\def\depth{\operatorname{depth}}
\def\im{\operatorname{Im}}
\def\ker{\operatorname{Ker}}

\def\nCM{\operatorname{nCM}}
\def\CM{\operatorname{CM}}
\def\Sing{\operatorname{Sing}}

\def\Hom{\operatorname{Hom}}
\def\htt{\operatorname{height}}

\newcommand{\fm}{\ensuremath{\mathfrak m}}
\newcommand{\fn}{\ensuremath{\mathfrak n}}
\newcommand{\fa}{\ensuremath{\mathfrak a}}
\newcommand{\fb}{\ensuremath{\mathfrak b}}
\newcommand{\fp}{\ensuremath{\mathfrak p}}
\newcommand{\fq}{\ensuremath{\mathfrak q}}
\newcommand{\pr}{\ensuremath{\prime}}

\newcommand{\cO}{\ensuremath{\mathcal O}}
\newcommand{\cR}{\ensuremath{\mathcal R}}

\newcommand{\cQ}{\ensuremath{\mathcal Q}}

\begin{document}


\title[Local cohomology annihilators and Macaulayfication]{Local cohomology annihilators and Macaulayfication}

\author{\fontencoding{T5} Nguyen Tu Cuong} 
\address{\fontencoding{T5} Nguyen Tu Cuong, 
Institute of Mathematics, 18 Hoang Quoc Viet, 10307 Hanoi, Vietnam.}
\email{ntcuong@math.ac.vn}

\author{\fontencoding{T5} Doan Trung Cuong}
\address{\fontencoding{T5} Doan Trung Cuong. {\it Current address:} Vietnam Institute for Advanced Study in Mathematics, Ta Quang Buu Building, 01 Dai Co Viet, Hai Ba Trung, Hanoi, Vietnam. {\it Permanent address:} Institute of Mathematics, 18 Hoang Quoc Viet, 10307 Hanoi, Vietnam.}
\email{doantc@gmail.com}

\thanks{Nguyen Tu Cuong is supported by the NAFOSTED of Vietnam under grant number 101.01-2011.49.}
\thanks{Doan Trung Cuong is partially supported by the SFB/TR 45 "Periods, moduli spaces and arithmetic of algebraic varieties" and by the NAFOSTED of Vietnam under grant number 101.01-2012.05.}


\subjclass[2010]{Primary 13H10,  14M05; Secondary 13D45, 14B05}
\keywords{arithmetic Macaulayfication, Macaulayfication, local cohomology annihilator, p-standard system of parameters, quotient of Cohen-Macaulay ring}
\begin{abstract}
The aim of this paper is to study a deep connection between local cohomology annihilators and Macaulayfication and arithmetic Macaulayfication over a local ring. Local cohomology annihilators appear through the notion of p-standard system of parameters. For a local ring, we prove an equivalence of the existence of Macaulayfications; the existence of a p-standard system of parameters; being a quotient of a Cohen-Macaulay local ring; and the verification of Faltings' Annihilator theorem. For a finitely generated module which is unmixed and faithful, we prove an equivalence of the existence of an arithmetic Macaulayfication and the existence of a p-standard system of parameters; and both are proved to be equivalent to the existence of an arithmetic Macaulayfication on the ground ring. A  connection between Macaulayfication and universal catenaricity is also discussed.

\end{abstract}
\maketitle

\section{Introduction}
Let $X$ be a Noetherian scheme. A Macaulayfication of $X$ is a pair $(Y, \pi)$ consisting of a Cohen-Macaulay scheme $Y$ and of a birational proper morphism $\pi: Y\rightarrow X$. This analogous notion of desingularization is due to Faltings and he constructed a Macaulayfication for any quasi-projective scheme with $0$ or $1$-dimensional non-Cohen-Macaulay locus over a Noetherian ring admitting a dualizing complex \cite{gf2}. The key point in Faltings' construction of Macaulayfication is a profound connection between Macaulayfication of a local ring and the annihilators of local cohomology modules of the ring. Using this idea, Kawasaki \cite{tk1} constructed a Macaulayfication for any quasi-projective scheme over a Noetherian ring provided the ground ring admits a dualizing complex. Local cohomology annihilators appears in Kawasaki's construction through the notion of p-standard system of parameters which is defined first in
 \cite{ntc, acta}. The Macaulayfication constructed by Kawasaki in fact is a blowing up of the scheme with center formulated by a product of parts of a p-standard system of parameters. While desingularization of Hironaka \cite{hi} depends on characteristic of the ground field, the constructed Macaulayfication is significantly characteristic-independent.

Beside Faltings and Kawasaki's Macaulayfication, there are other construction in some special cases by Brodmann \cite{br2}, Goto \cite{go1}, Schenzel \cite{sch1}. There is also a formulation of Macaulayfication of modules and sheaves \cite{mo}, \cite{br3}.

Going further, Kawasaki \cite{tk2, tk3} characterized the existence of arithmetic Macaulayfication of a Noetherian ring. Let $R$ be a commutative Noetherian ring and $I$ be a proper ideal of positive height of $R$. The Rees algebra $\cR(I)=\bigoplus_{m=0}^\infty I^m$ gives rise to a blowing up $Y=\Proj \cR(I) \stackrel{\pi}{\rightarrow} \Spec(R)$. The morphism $\pi$ is a Macaulayfication if $Y$ is a Cohen-Macaulay scheme. If the Rees algebra $\cR(I)$ itself is Cohen-Macaulay then it is called an arithmetic Macaulayfication of the ring $R$. Obviously the proj of an arithmetic Macaulayfication is a Macaulayfication of the spectrum of the ring. Kawasaki showed that a Noetherian local ring has an arithmetic Macaulayfication if and only if it is unmixed and all its formal fibers are Cohen-Macaulay. It is also remarkable that the ideal defining the Cohen-Macaulay Rees algebra is also a product of parts of a p-standard system of parameters. 

Arithmetic Macaulayfication has been studied from other perspective by Kurano \cite{ku}, Aberbach \cite{ab}, Aberbach-Huneke-Smith \cite{ahk}, Cutkosky-Tai \cite{ct}, Tai-Trung \cite{tt}. Analogously, an arithmetic Macaulayfication of a finitely generated $R$-module $M$ is defined to be a Cohen-Macaulay Rees module $\cR(M, I):=\bigoplus_{m\geq 0}I^mM$ for some ideal $I$ of $R$.

The aim of this paper is to study more extensively the relationship between Macaulayfication and local cohomology annihilators. We address ourself to two problems on local rings: Existence of arithmetic Macaulayfication of modules; relationship between existence of p-standard system of parameters and of Macaulayfication of the spectrum of a local ring.

Let $(R, \fm)$ be a Noetherian local ring and $M$ be a finitely generated $R$-module. 
The $i$-th local cohomology module $H^i_\fm(M)$ has the annihilator ideal denoted by $\fa_i(M)$. Put $\fa(M):=\fa_0(M)\fa_1(M)\ldots \fa_{d-1}(M)$, where $d$ is the Krull dimension of $M$. A system of parameters $x_1, \ldots, x_d\in \fm$ of $M$ is called p-standard if for any $i=1, \ldots, d$, we have $x_i\in \fa(M/(x_{i+1}, \ldots, x_d)M)$. These systems of parameters have very rich properties. Beside applications in constructing Macaulayfication, they are very useful in the study of structure of local ring and modules, see Cuong-Schenzel-Trung \cite{cst}, Schenzel \cite{sch}, Trung \cite{tr}, Cuong-Cuong \cite{cc1, cc2}. p-Standard system of parameters will play an important role in our study of two problems above.

For the problem on arithmetic Macaulayfication of modules, we get the first main result. 

\begin{theorem} \label{main1a} Let $(R, \fm)$ be a Noetherian local ring and $M$ be a finitely generated $R$-module.  Suppose $M$ is unmixed. The following statements are equivalent:

\begin{enumerate}[(a)]
\item $M$ has an arithmetic Macaulayfication.
\item  $R/\fp$ has an arithmetic Macaulayfication for all associated prime ideals $\fp\in \Ass(M)$.
\item $M$  admits a p-standard system of parameters.
\end{enumerate}
\end{theorem}

Here $M$ is unmixed if $\dim \hat R/P=\dim M$ for any associated prime ideal $P$ of the $\fm$-adic completion of $M$. 

As a direct consequence of Theorem \ref{main1a}, a finitely generated $R$-module $M$ has an arithmetic Macaulayfication if and only if so does the ring $R/\Ann_R(M)$. So the question on existence of arithmetic Macaulayfication of module reduces to the same question on the corresponding ring. In order to prove Theorem \ref{main1a}, we first prove the equivalence of (a) and (c). Then we make use of the following result which relates the existence of p-standard system of parameters on a module to that on the ring.

\begin{theorem} \label{main2a} Let $(R, \fm)$ be a Noetherian local ring and $M$ be a finitely generated $R$-module. The following statements are equivalent:

\begin{enumerate}[(a)]
\item $M$ admits a p-standard system of parameters.
\item $R/\Ann_RM$ admits a p-standard system of parameters. 
\item Any finitely generated $R$-module $N$ with $\Supp(N)\subseteq \Supp(M)$ admits a p-standard system of parameters.
\end{enumerate}
\end{theorem}
An interesting consequence of Theorems \ref{main1a} and \ref{main2a} is that if $M$ has an arithmetic Macaulayfication and $N$ is another finitely generated $R$-module such  that $N$ is unmixed and $\Supp_R(N)\subseteq \Supp_R(M)$, then $N$ has an arithmetic Macaulayfication.

In the second part of the paper, we study the existence of p-standard system of parameters on a local ring. It turns out that the existence of p-standard system of parameters relates closely to various subjects such as Macaulayfication, Faltings' Annihilator theorem and being quotient of Cohen-Macaulay rings. They are put all together in the following second main result of the paper.

\begin{theorem}\label{main3a} Let $(R, \fm)$ be a Noetherian local ring. The following statements are equivalent:
\begin{enumerate}[(a)]
\item $R$  admits a p-standard system of parameters.
\item $R$ is a quotient of a Cohen-Macaulay local ring.
\item  $R$ is universally catenary and for any quotient $S$ of $R$,  $\Spec(S)$ has a Macaulayfication.
\item All essentially of finite type $R$-algebras verify Faltings' Annihilator Theorem. That means, if we let $A$ be an  essentially of finite type $R$-algebra and $N$ be a finitely generated $A$-module, then for any pair of ideals $\fb\subseteq \fa$ of $A$, we have
$$\inf \{i: \fb\not\subseteq \sqrt{\Ann_A H^i_\fa(N)}\}=\inf\{\depth N_\fp+\htt(\fa+\fp)/\fp: \fp\notin V(\fb)\}.$$
\item Any of the above statements holds for all essentially of finite type $R$-algebras.
\end{enumerate}
\end{theorem}

Theorem \ref{main3a} shows the existence of Macaulayfication of the spectrum of a local ring provided this local ring is quotient of  a Cohen-Macaulay ring. This result in fact a generalization of Kawasaki's Theorem on Macaulayfication \cite[Theorem 1.1]{tk1} for local rings. Furthermore, by Theorems \ref{main1a} and \ref{main3a}, a local ring is a quotient of a Cohen-Macaulay ring if and only if there is a finitely generated and faithful module over the ring who has a p-standard system of parameters. 

The key point in the proof of Theorem \ref{main3a} is to show that a local ring has a p-standard system of parameters if and only if it is universally catenary and all its formal fibers are Cohen-Macaulay. Then we use several times results of Faltings, Grothendieck and Kawasaki to relate the later properties with the other equivalent conditions in the theorem.
\medskip

To summarize the content of this paper, in Section \ref{section2} we study the local cohomology supported by an ideal generated by a dd-sequence, a slightly generalized notion of p-standard system of parameters (see Lemma \ref{lem0}). p-Standard system of parameters is defined via the local cohomology annihilators while dd-sequence is defined by using Huneke's notion of d-sequence. We will see in Section \ref{section3} that in certain cases, dd-sequences as sequences of elements in $\fm$ is easier to handle than p-standard sytem of parameters. The main result of this section is a splitting short exact sequence for local cohomology modules in Theorem \ref{thm12}. 

The main technical results for the proofs of Theorems \ref{main1a} and \ref{main2a} are presented in Section \ref{section3}. Using results in Section \ref{section2}, we show that the existence of p-standard system of parameters is a local property. That means, if a module has a p-standard system of parameters then so does its localization at any prime ideal (Proposition \ref{pr5}). 

We prove Theorem \ref{main1a} in Section \ref{section4}. We first prove Theorem \ref{main2a} where the technical results in the former sections find their applications.

Theorem \ref{main3a} will be proved in both Sections \ref{section5} and \ref{section6}. In Section \ref{section5} the equivalence of the conditions (a) and (b) in Theorem \ref{main3a} is proved by Theorem \ref{main11}. Various consequences of this theorem characterizing quotients of Cohen-Macaulay local rings are presented afterward. The last section is involved to prove the rest of Theorem \ref{main3a} (see Theorem \ref{main12}) and to discuss a conjectural relationship between Macaulayfication (originally, desingularization) and universal catenaricity of a Noetherian ring (Theorem \ref{catenary}). This was motivated by Grothendieck's works on connection between desingularization and excellent rings in \cite{gr}.


\section{Local cohomology supported by dd-sequences} \label{section2}
Throughout this paper $(R, \fm)$ denotes a commutative Noetherian local ring. 

 In this section we study the local cohomology with support generated by a dd-sequence and  their annihilator ideals. Let $M$ be a finitely generated $R$-module of dimension $d$. Denote by $\fa_i(M)=\Ann_R H^i_\fm(M)$ the annihilator ideal of the $i$-th local cohomology module $H^i_\fm(M)$ and set $\fa(M)=\fa_0(M)\ldots\fa_{d-1}(M)$. Due to Huneke \cite{hu}, a sequence $x_1, \ldots, x_s\in \fm$ is  a d-sequence on $M$ if $(x_1, \ldots, x_i)M:x_j=(x_1, \ldots, x_i)M:x_{i+1}x_j$ for all $0\!
 \leq i<j\leq s$. 

\begin{definition} A system of parameters $x_1, \ldots, x_d$ of $M$ is called {\it p-standard} if $x_d\in \fa(M)$ and $x_i\in \fa(M/(x_{i+1}, \ldots, x_d)M)$ for $i=d-1, \ldots, 1$ (cf. \cite{acta}). In \cite{cc}, the authors extend this notion for sequence of elements. We say that a sequence $x_1, \ldots, x_s $ in $\frak m$ is a {\it dd-sequence} on $M$ if $x_1^{n_1}, \ldots, x_i^{n_i} $ is a d-sequence on $M/(x_{i+1}^{n_{i+1}}, \ldots, x_s^{n_s})M$ for all $n_1, \ldots, n_s>0$ and $i=1, \ldots, s$. 
\end{definition}

The two notions above are very close in case of system of parameters as in the following lemma.

\begin{lemma}\cite[Corollary 3.9]{cc}\label{lem0}
Let $M$ be a finitely generated $R$-module with a system of parameters $x_1, \ldots, x_d$, $d=\dim M$. If $x_1, \ldots, x_d$ is a p-standard system of parameters then it is a dd-sequence on $M$. Vice versa, if $x_1, \ldots, x_d$ is a dd-sequence then $x_1^{n_1}, \ldots, x_d^{n_d}$ is a p-standard system of parameters of $M$ for all $n_i\geq i$, $i=1, \ldots, d$.
\end{lemma}

Belows we recall some basic properties of dd-sequences. 

\begin{lemma}\cite[Proposition 3.4]{cc}\label{lem1} 
Let $M$ be a finitely generated $R$-module and $x_1, \ldots, x_s$ be a dd-sequence on $M$. 

\item(i) $x_{i_1}, \ldots, x_{i_r}$ is a dd-sequence on $M$ for any $1\leq i_1<\ldots<i_r\leq s$.

\item(ii) $x_1, \ldots, x_{i-1}, x_{i+1}, \ldots, x_s$ is a dd-sequence on $M/x_iM$ for $i=1, \ldots, s$.
\end{lemma}

As usual we denote a sequence $x_1, \ldots, x_{i-1}, x_{i+1}, \ldots, x_s$ by $x_1, \ldots, \widehat{x_i}, \ldots, x_s$. By a standard method for d-sequences and Koszul homology (see Goto-Yamagishi \cite[Theorem 1.14]{gy}) and using the same argument as in the proof of Theorem 2.3 of Goto-Yamagishi  \cite{gy} we have
\begin{lemma}\label{lem2} Let $x_1, \ldots, x_s$ be a dd-sequence on $M$. It holds

\item(i) $(x_{r+1}, \ldots, x_s)H^i_{(x_1, \ldots, x_r)}(M)=0$, for $ i<r, 1\leq i\leq s$.

\item(ii) $(x_1^{n_1+m_1}, \ldots, x_s^{n_s+m_s})M:(x_1^{m_1}\ldots x_s^{m_s})$
$$\hspace{4cm}=\sum_{i=1}^s(x_1^{n_1}, \ldots, \widehat{x_i^{n_i}},\ldots, x_s^{n_s})M:x_i+(x_1^{n_1}, \ldots, x_s^{n_s})M.$$
\end{lemma}
\begin{lemma}\label{pr8}
Let $x_1, \ldots, x_s$ be a dd-sequence on $M$. For $i<s$ we have
$$H^i_{(x_1, \ldots, x_s)}(M)\simeq H^i_{(x_1, \ldots, x_{i+1})}(M)\simeq (0:x_{i+1})_{H^i_{(x_1, \ldots, x_i)}(M)}.$$
\end{lemma}
\begin{proof}
Since $x_1, \ldots, x_k$, $1\leq k\leq s$, is also a dd-sequence on $M$ by Lemma \ref{lem1}, it suffices to prove that
\begin{enumerate}[(i)]
\item $H^i_{(x_1, \ldots, x_s)}(M)\simeq H^i_{(x_1, \ldots, x_{s-1})}(M)$.

\item $H^{s-1}_{(x_1, \ldots, x_s)}(M)\simeq (0:x_s)_{H^{s-1}_{(x_1, \ldots, x_{s-1})}(M)}.$
\end{enumerate}

\noindent Denote the $i$-th $\check{\mathrm C}$ech cohomology module of $M$ with respect to the ideal $(x_1, \ldots, x_s)$ by $\check H^i(x_1, \ldots, x_s, M)$. There is a long exact sequence
\begin{multline*}\ldots \longrightarrow \check H^i(x_1, \ldots, x_s, M)\longrightarrow\check H^i(x_1, \ldots, x_{s-1}, M)\\
\longrightarrow \check H^i(x_1, \ldots, x_{s-1}, M_{(x_s)})\longrightarrow \check H^{i+1}(x_1, \ldots, x_s, M)\longrightarrow \ldots,
\end{multline*}
where $M_{(x_s)}$ is the localization of $M$ by the multiplicatively closed set $\{1, x_s, x_s^2, \ldots\}$. Since $\check H^i(x_1, \ldots, x_s, M)\simeq H^i_{(x_1, \ldots, x_s)}(M)$ and $\check H^i(x_1, \ldots, x_{s-1}, M_{(x_s)})\simeq $\linebreak $ H^i_{(x_1, \ldots, x_{s-1})}(M)_{(x_s)}$, there is a long exact sequence of local cohomology modules
\begin{multline*}\ldots \longrightarrow H^i_{(x_1, \ldots, x_s)}(M)\longrightarrow H^i_{(x_1, \ldots, x_{s-1})}(M)
\longrightarrow H^i_{(x_1, \ldots, x_{s-1})}(M)_{(x_s)}\\\longrightarrow H^{i+1}_{(x_1, \ldots, x_s)}(M)\longrightarrow \ldots
\end{multline*}
Remind that $x_1, \ldots, x_s$ is a dd-sequence on $M$, so $x_s\in \Ann_R H^i_{(x_1, \ldots, x_{s-1})}(M)$ for all $i<s-1$ by Lemma \ref{lem2}($i$) and $H^i_{(x_1, \ldots, x_s)}(M)_{(x_s)}=0$. Therefore, $H^i_{(x_1, \ldots, x_s)}(M)\simeq H^i_{(x_1, \ldots, x_{s-1})}(M)$ for all $i<s-1$ and there is an exact sequence
\begin{multline*}0 \longrightarrow H^{s-1}_{(x_1, \ldots, x_s)}(M)\longrightarrow H^{s-1}_{(x_1, \ldots, x_{s-1})}(M)
\stackrel{\psi}{\longrightarrow} H^{s-1}_{(x_1, \ldots, x_{s-1})}(M)_{(x_s)}\\\longrightarrow H^s_{(x_1, \ldots, x_s)}(M)\longrightarrow 0,
\end{multline*}
where $\psi$ is the natural homomorphism of localization. Note that we can identify $H^{s-1}_{(x_1, \ldots, x_{s-1})}(M)= \underset{t}{\varinjlim} M/(x_1^t, \ldots, x_{s-1}^t)M$. Let $\bar a\in H^{s-1}_{(x_1, \ldots, x_{s-1})}(M)$, where $a\in$\linebreak $M/(x_1^t, \ldots, x_{s-1}^t)M$ for some $t>0$. So $\bar a\in \ker \psi$ if and only if $x_s^r\bar a=0$ for some $r>0$ if and only if
$x_s^ra(x_1\ldots x_{s-1})^t\in (x_1^{t+t^\pr},\ldots, x_{s-1}^{t+t^\pr})M$. Equivalently,
\[\begin{aligned}a\in  (x_1^{t+t^\pr},\ldots, x_{s-1}^{t+t^\pr})M:x_s^r(x_1\ldots x_{s-1})^{t^\pr}&=(x_1^{t+t^\pr},\ldots, x_{s-1}^{t+t^\pr})M:x_s(x_1\ldots x_{s-1})^{t^\pr}\\&=(x_1^{t+1},\ldots, x_{s-1}^{t+1})M:(x_1\ldots x_{s-1}x_s)\end{aligned}\]
by Lemma \ref{lem2}($ii$) and the assumption $x_1,\ldots, x_s$ is a dd-sequence. So $x_s\bar a=0$. Therefore,
$$H^{s-1}_{(x_1, \ldots, x_s)}(M)\simeq \ker \psi\simeq (0:x_s)_{H^{s-1}_{(x_1, \ldots, x_{s-1})}(M)}.$$
\end{proof}
\begin{corollary}\label{co9pr}
Let $x_1, \ldots, x_s$ be a dd-sequence on $M$. We have
$$\Ann_R H^i_{(x_1, \ldots, x_s)}(M)=\bigcap_{n_1, \ldots, n_i>0}\Ann \frac{(x_1^{n_1}, \ldots, x_i^{n_i})M:x_{i+1}}{(x_1^{n_1+1}, \ldots, x_i^{n_i+1})M:(x_1\ldots x_i)}$$
for $i<s$.
\end{corollary}
\begin{proof}
From Lemma \ref{pr8}, 
$$H^i_{(x_1, \ldots, x_s)}(M)\simeq (0:x_{i+1})_{H^i_{(x_1, \ldots, x_i)}(M)}=\underset{t}{\varinjlim} (0:x_{i+1})_{M/(x_1^t, \ldots, x_i^t)M}.$$
Hence, $a\in \Ann_R H^i_{(x_1, \ldots, x_s)}(M)$ if and only if for any $t>0$ and $u\in (x_1^t, \ldots, x_i^t)M:x_{i+1}$, there is $t^\pr>0$ such that $au(x_1\ldots x_i)^{t^\pr}\in (x_1^{t+t^\pr}, \ldots, x_i^{t+t^\pr})M$, or equivalently, 
\[\begin{aligned}au&\in (x_1^{t+t^\pr}, \ldots, x_i^{t+t^\pr})M:(x_1\ldots x_i)^{t^\pr}\\
&=\sum_{j=1}^i(x_1^{t+1}, \ldots, \widehat{x_j^{t+1} }, \ldots, x_i^{t+1})M:x_j+(x_1^t, \ldots, x_i^t)M\\
&=(x_1^{t+1}, \ldots, x_i^{t+1})M:(x_1\ldots x_i),\end{aligned}\]
by Lemma \ref{lem2}($ii$). In other words, $a\in \Ann_R H^i_{(x_1, \ldots, x_s)}(M)$ if and only if 
$$a[(x_1^t, \ldots, x_i^t)M:x_{i+1}]\subseteq (x_1^{t+1}, \ldots, x_i^{t+1})M:(x_1\ldots x_i),$$
for any $t>0$. Since $H^i_{(x_1, \ldots, x_s)}(M)\simeq H^i_{(x_1^{n_1}, \ldots, x_s^{n_s})}(M)$ for any $n_1, \ldots, n_s>0$, we can replace $x_1, \ldots, x_i$ by $x_1^{n_1}, \ldots, x_i^{n_i}$ to obtain
$$\Ann_R H^i_{(x_1, \ldots, x_s)}(M)=\bigcap_{n_1, \ldots, n_i>0}\Ann \frac{(x_1^{n_1}, \ldots, x_i^{n_i})M:x_{i+1}}{(x_1^{n_1+1}, \ldots, x_i^{n_i+1})M:(x_1\ldots x_i)}$$
\end{proof}
The next theorem gives a remarkable splitting short exact sequence of local cohomology modules, this exact sequence will be useful later on (see also \cite[Proposition 2.6]{dc}).
\begin{theorem}\label{thm12}
Let $(R, \fm)$ be a Noetherian local ring, $I\subseteq R$ be an ideal and $M$ be a finitely generated $R$-module. Fix a non-negative integer $k$. Assume that there is $x\in \fm$ such that $0:_Mx=0:_Mx^2$ and $xH^i_I(M)=0$ for all $i\leq k$. Let $N\subseteq 0:_Mx$ be a submodule. There is a splitting short exact sequence of local cohomology modules
$$0\longrightarrow H^i_I(M/N)\longrightarrow H^i_I(M/x^nM+N)\longrightarrow H^{i+1}_I(M/0:_Mx)\longrightarrow 0,$$
for all $i<k$  and $n\geq 5$. In particular,
$$H^i_I(M/x^nM+N)\simeq  H^i_I(M/N)\oplus H^{i+1}_I(M/0:_Mx).$$
\end{theorem}
\begin{proof}
We first imply that $x^2H^i_I(M/0:_Mx)=0$ for all $i\leq k$ from the long exact sequence of local cohomology modules
$$\ldots \longrightarrow H^i_I(0:_Mx)\longrightarrow H^i_I(M)\longrightarrow H^i_I(M/0:_Mx)\longrightarrow H^{i+1}_I(0:_Mx)\longrightarrow \ldots.$$
Since $0:_Mx^n=0:_Mx$, there is a commutative diagram
\[\begin{CD}
0@>>>M/0:_Mx  @>.x^n>> M/N@>>> M/x^nM+N@>>>0\\
&&@V.x^2VV  @| @VpVV\\
0@>>>M/0:_Mx @>.x^{n-2}>> M/N@>>> M/x^{n-2}M+N@>>>0,
\end{CD}\]
where $p$ is the natural projection. The above diagram derives the following commutative diagram
\[\begin{CD}
\cdots\longrightarrow H_I^i(M/0:_Mx) @>\psi_i>>   H_I^i(M/N) @>>> H_I^i(M/x^nM+N)\longrightarrow\cdots\\
@V.x^2VV @|@VVV\\
\cdots\longrightarrow H_I^i(M/0:_Mx) @>\varphi_i>> H_I^i(M/N) @>>> H_I^i(M/x^{n-2}M+N)\longrightarrow\cdots,
\end{CD}\]
where $\psi_i, \varphi_i$ are homomorphisms derived from the homomorphisms $M/0:_Mx \stackrel{.x^n}{\longrightarrow} M/N$ and $M/0:_Mx  \stackrel{.x^{n-2}}{\longrightarrow} M/N$ respectively. If $n\geq 3$ then $\psi_i=\varphi_i\circ(.x^2)=0$. So if $n\geq 5$ then $\psi_i=0$ and $\varphi_i=0$ for all $i<k$. We obtain a commutative diagram with exact horizontal lines
\[\begin{CD}
0\longrightarrow  H_I^i(M/N) @>f_1>> H_I^i(M/x^nM+N)@>g_1>> H_I^{i+1}(M/0:_Mx) \longrightarrow 0\\
@|@VVhV@VV.x^2V \\
0\longrightarrow  H_I^i(M/N) @>f_2>> H_I^i(M/x^{n-2}M+N)@>g_2>> H_I^{i+1}(M/0:_Mx)\longrightarrow 0.
\end{CD}\]
Note that $g_2\circ h=(.x^2)\circ g_1=0$, thus $\im(h)\subseteq \im(f_2)$. We then get a homomorphism
$$f_2^{-1}\circ h: H^i_I(M/x^nM+N)\longrightarrow H^i_I(M/N),$$
which satisfies $f_2^{-1}\circ h\circ f_1=\mathrm{id}_{H^i_I(M/N)}$. Therefore, the following short exact sequence is splitting
$$0\longrightarrow  H_I^i(M/N) \longrightarrow H_I^i(M/x^nM+N)\longrightarrow H_I^{i+1}(M/0:_Mx) \longrightarrow 0,$$
for all $i<k$ and $n\geq 5$.
\end{proof}
\begin{corollary}\label{co13}
Let $x_1, \ldots, x_s$ be a dd-sequence on $M$. There is a splitting short exact sequence
$$0\longrightarrow  H_{(x_1, \ldots, x_s)}^i(M) \longrightarrow H_{(x_1, \ldots, x_s)}^i(M/x_j^nM)\longrightarrow H_{(x_1, \ldots, x_s)}^{i+1}(M/0:_Mx_j) \longrightarrow 0,$$
for $0\leq i<j-1<s$ and $n\geq 5$. In particular, if in addition $s=d$ and $x_1,\ldots, x_d$ is a system of parameter of $M$ then there is a splitting short exact sequence
$$0\longrightarrow  H_\fm^i(M) \longrightarrow H_\fm^i(M/x_j^nM)\longrightarrow H_\fm^{i+1}(M/0:_Mx_j) \longrightarrow 0,$$
for all $0\leq i<j-1<d$ and $n\geq 5$.
\end{corollary}
\begin{proof}
We have $H_{(x_1, \ldots, x_s)}^i(M)\simeq H_{(x_1, \ldots, x_{i+1})}^i(M)$ from Lemma \ref{pr8}. Following Lemma \ref{lem2}($i$), we have $x_jH_{(x_1, \ldots, x_s)}^i(M)=0$ for all $j> i+1$. The conclusion then follows from Theorem \ref{thm12}.
\end{proof}
Recall that $\fa_i(M):=\Ann_R H^i_\fm(M)$ and $\fa(M)=\fa_0(M)\ldots \fa_{d-1}(M)$, $d=\dim M$. The next lemma is a simple but very useful property of the ideal $\fa(M)$.
\begin{lemma}\label{am}
Let $M$ be a finitely generated $R$-module. For any submodule $N\subset M$ such that $\dim N<\dim M$, $\fa(M)\subseteq \Ann_R N$. In particular, $\fa(M)\subseteq \fp$ for all $\fp\in \Ass(M)$ with $\dim R/\fp<\dim M$.
\end{lemma}
\begin{proof}
Since $\dim N<\dim M$, there is a parameter element $x$ of $M$ such that $x\in \Ann_RN$. Hence, $N\subseteq 0:_Mx$. Note that by a result of Schenzel \cite[Satz 2.4.5]{sch}, $\fa(M) \subseteq \Ann_R(0:_Mx)$, this implies the first conclusion.

Let $\fp\in \Ass(M)$ with $\dim R/\fp<\dim M$. There is a submodule $N\subset M$ such that $\Ass(N)=\{\fp\}$. Thus $\dim N=\dim R/\fp<\dim M$ and by the first conclusion, $\fa(M)\subseteq \Ann_R N\subseteq \fp$.
\end{proof}
We end this section by the following corollary of Theorem \ref{thm12}.
\begin{corollary}\label{co13prpr}
Let $x, y\in \fa(M)$ be two parameter elements of $M$. For any $n, m\geq 5$ and $i=0, 1, \ldots, d-2$ we have 
$$H^i_\fm(M/x^nM)\simeq H^i_\fm(M/y^mM).$$
\end{corollary}
\begin{proof}
Since $x, y\in \fa(M)$, we have $x\in \Ann_R (0:_My)$ and $y\in \Ann_R (0:_Mx)$ by Lemma \ref{am}. Thus  $0:_My= 0:_Mx$, in particular, $0:_Mx^t= 0:_Mx$ and $0:_My^t= 0:_My$ for any $t>0$. Using Theorem \ref{thm12} we obtain for $n, m\geq 5$, $i=0, 1, \ldots, d-2$,
\begin{multline*}
H^i_\fm(M/x^nM)\simeq H^i_\fm(M)\oplus H^{i+1}_\fm(M/0:_Mx)\\
\simeq H^i_\fm(M)\oplus H^{i+1}_\fm(M/0:_My)\simeq H^i_\fm(M/y^mM).
\end{multline*}
\end{proof}


\section{Localization} 
\label{section3}

In general a finitely generated module over a local ring does not necessarily have a p-standard system of parameters. A typical example is the two dimensional local domain given by Ferrand-Raynaud \cite[Proposition 3.3]{fr}. There are only some sufficient conditions for the existence of such a system of parameters, for examples, when the ground ring admits a dualizing complex (cf. \cite{ntc}). Other examples are Cohen-Macaulay modules and generalized Cohen-Macaulay modules, namely, regular and standard systems of parameters respectively. Recently, the authors have shown that this is also the case of sequentially Cohen-Macaulay modules and sequentially generalized Cohen-Macaulay modules. In this section, we will study the existence of p-standard system of parameters when passing to localization, this is a key for the proofs of Theorems \ref{main1} and \ref{main11} in next sections. To do this, we need first some lemmas.
\begin{lemma}\label{lem3.1}
Let $M$ be a finitely generated $R$-module of dimension $d$ and let $x$ be a parameter element of $M$. Then $\fa^{d-1}(M)\subseteq \fa(M/xM)$.
\end{lemma}
\begin{proof}
Set $\fb(M)=\bigcap_{i=1}^d\Ann (0:x_i)_{M/(x_1, \ldots, x_{i-1})M}$ where $x_1, \ldots, x_d$ runs over the set of all systems of parameters of $M$. By Schenzel \cite[Satz 2.4.5]{sch}, we have
$$\fa(M)\subseteq \fb(M)\subseteq \fa_0(M)\cap \ldots \cap \fa_{d-1}(M).$$
Hence, $\fa^{d-1}(M)\subseteq \fb^{d-1}(M)\subseteq \fb^{d-1}(M/xM)\subseteq\fa(M/xM)$.
\end{proof}
\begin{lemma}\label{lem3.2}
Let $M$ be a finitely generated $R$-module. Assume that $M$ admits a p-standard system of parameters $x_1, \ldots, x_d$. Let $y_d\in \fa(M)$ be a parameter element of $M$. There are $y_1, \ldots, y_{d-1}\in \fm$ such that $y_1^n, \ldots, y_d^n$ is a p-standard system of parameters of $M$ for all $n\gg 0$.
\end{lemma}
\begin{proof}
Take $n\geq 5. d!$. We choose $y_k$ by induction on $r=d-k$. The case $r=0$ is trivial. Assume that $r>0$ and we have chosen $y_{k+1}, \ldots, y_d$ such that
\begin{enumerate}[(i)]
\item $y_{k+1}, \ldots, y_i, x_{i+1}, \ldots, x_d$ is  a part of a system of parameters for $i=k+1, \ldots, d$.

\item $y_i\in \fa(M/(y_{i+1}^n, \ldots, y_d^n)M)$ for $i=k+1, \ldots, d$.
\end{enumerate}

\noindent From Lemma \ref{lem3.1}, $x_d^n, y_d^n\in \fa(M/(y_i^n, \ldots, y_{d-1}^n)M)$, $i=k+1, \ldots, d-1$. This implies by Corollary \ref{co13prpr} that
$$H^j_\fm(M/(y_i^n, \ldots, y_{d-1}^n, y_d^n)M)\simeq H^j_\fm(M/(y_i^n, \ldots, y_{d-1}^n, x_d^n)M),$$
for $j=0, 1, \ldots, i-2$. The module $M/x_d^nM$ admits a p-standard system of parameters $x_1, \ldots, x_{d-1}$ and $y_{k+1}, \ldots, y_{d-1}$ is a part of a system of parameters of it such that 
$$y_i\in \fa(M/(y_{i+1}^n, \ldots, y_{d-1}^n, x_d^n)M),$$
for $ i=k+1, \ldots, d-1$. Using the induction hypothesis, we can choose $$y_k\in \fa(M/(y_{k+1}^n, \ldots, y_{d-1}^n, x_d^n)M)=\fa(M/(y_{k+1}^n, \ldots, y_{d-1}^n, y_d^n)M),$$
such that $y_k\not\in\fp$ for all $\fp\in \Ass M/(y_{k+1}, \ldots, y_i, x_{i+1}, \ldots, x_d)M$ with $\dim R/\fp=k$, $i=k+1, \ldots, d$.
\end{proof}

In the next we will show that the existence of p-standard system of parameters is preserved when passing to localization. A technical remark is, in some cases dd-sequences have more advantage than p-standard systems of parameters, one can see this in the proofs of the previous lemmas and of the following proposition.
\begin{proposition}\label{pr5}
Let $R$ be a Noetherian local ring and $M$ be a finitely generated $R$-module. Let $\fp\in \Supp(M)$ and set $r=\dim R/\fp$. Assume in addition that $M$ admits a p-standard system of parameters. There is a p-standard system of parameters $x_1, \ldots, x_d$ of $M$ such that $x_{r+1}, \ldots, x_d\in \fp$ and if we put $s_1=\dim M_\fp, s_2=\depth M_\fp$, then

(i) $x_{r+1}, \ldots, x_{r+s_1}$ is a p-standard system of parameters of $M_\fp$.

(ii) $x_{r+1}, \ldots, x_{r+s_2}$ is a maximal regular sequence on $M_\fp$.

\noindent In particular, $M_\fp$ admits a p-standard system of parameters.
\end{proposition}
\begin{proof}
If $r<d$ then we choose first a parameter element $x_d\in \fp\cap\fa(M)$. Let $n>5. d!$. By Lemma \ref{lem3.2}, $M/x_d^nM$ also admits a p-standard system of parameters. If $r<d-1$ we can choose a parameters element $x_{d-1}\in \fp\cap\fa(M/x_d^nM)$ of $M/x_d^nM$. By the recursive method, we get a system of parameters $x_1,\ldots, x_d$ of $M$ such that $x_{r+1}, \ldots, x_d\in \fp$ and $x_i\in\fa(M/(x_{i+1}^n, \ldots, x_d^n)M)$, $i=1, \ldots, d$. Then $x_1^n, \ldots, x_d^n$ is a p-standard system of parameters of $M$ with $x_{r+1}, \ldots, x_d\in \fp$. Without lost of generality we assume that $x_1, \ldots, x_d$ is a p-standard system of parameters. By Lemma \ref{lem0}, $x_1, \ldots, x_d$ is a dd-sequence on $M$, hence $x_{r+1}, \ldots, x_d$ is a dd-sequence on $M$ by Lemma \ref{lem1}($i$). Passing to localization, it can be seen easily that $x_{r+1}, \ldots, x_d$ is also a dd-sequence on $M_\fp$ for any prime ideal $\fp\in\Spec(R)$. Moreover, we have $\sqrt{(x_{r+1}, \ldots, x_d)R_\fp}=\fp R_\fp$ and
$$0:_{M_\fp}x_{r+1}=\bigcup_{t=1}^\infty0:_{M_\fp}(x_{r+1}^t, \ldots, x_d^t)=H^0_{\fp R_\fp}(M_\fp).$$
So if $\dim M_\fp >0$ then $x_{r+1}$ is a parameter element of $M_\fp$. Using Lemma \ref{lem1}($ii$), $x_1, \ldots, \widehat{x_{r+1}}, \ldots, x_d$ is a dd-sequence on $M/x_{r+1}M$. Induction on the dimension of $M$ shows that $x_{r+1}, \ldots, x_{r+s_1}$ is a system of parameters of $M_\fp$ which is also a dd-sequence on $M_\fp$. Using Lemma \ref{lem0} again we see that for $t$ big enough, $x_1^t, \ldots, x_d^t$ is a p-standard system of parameters of $M$ with $x_{r+1}, \ldots, x_d\in \fp$ and $x_{r+1}^t, \ldots, x_{r+s_1}^t$ is a p-standard system of parameters of $M_\fp$.

The regularity of the sequence $x_{r+1}, \ldots, x_{r+s_2}$ on $M_\fp$ is proved similarly.
\end{proof}
\begin{corollary}\label{co6}
Keep all assumptions in the previous proposition. For $k\geq \dim R/\fp$, we have
$$\fa_k(M)R_\fp\subseteq \fa_{k-\dim R/\fp}(M_\fp).$$
\end{corollary}
\begin{proof}
Set $r=\dim R/\fp$. We have shown in Proposition \ref{pr5} the existence of a p-standard system of parameters $x_1, \ldots, x_d$ of $M$ such that $x_{r+1}, \ldots, x_d\in \fp$. By Corollary \ref{co9pr}, 
$$\fa_k(M)[(x_1^{n_1}, \ldots, x_k^{n_k})M:x_{k+1}]\subseteq (x_1^{n_1+1}, \ldots, x_k^{n_k+1})M:(x_1\ldots x_k),$$
for $k=0,1, \ldots, d$ and $x_{d+1}=0$. Applying Krull's Intersection Theorem we get
\[\begin{aligned}
\fa_k(M)[(x_{r+1}^{n_{r+1}}, \ldots, x_k^{n_k})M:x_{k+1}]
&=\bigcap_{n_1, \ldots, n_r}\fa_k(M)[(x_1^{n_1}, \ldots, x_k^{n_k})M:x_{k+1}]\\
&\subseteq \bigcap_{n_1, \ldots, n_r}(x_1^{n_1+1}, \ldots, x_k^{n_k+1})M:(x_1\ldots x_k)\\
&=(x_{r+1}^{n_{r+1}+1}, \ldots, x_k^{n_k+1})M:(x_1\ldots x_k).
\end{aligned}\]
Since $x_1, \ldots, x_r\not\in\fp$, the above inclusion becomes
$$\fa_k(M)R_\fp[(x_{r+1}^{n_{r+1}}, \ldots, x_k^{n_k})M_\fp:x_{k+1}]
\subseteq (x_{r+1}^{n_{r+1}+1}, \ldots, x_k^{n_k+1})M_\fp:(x_{r+1}\ldots x_k).$$
It should be noted that $x_{r+1}, \ldots, x_d$ is a dd-sequence on $M_\fp$ and $\sqrt{(x_{r+1}, \ldots, x_d)R_\fp}=\fp R_\fp$. So by Corollary \ref{co9pr} again,
$$\fa_{k-r}(M_\fp)=\bigcap_{n_{r+1}, \ldots, n_k>0}\Ann \frac{(x_{r+1}^{n_{r+1}}, \ldots, x_k^{n_k})M_\fp:x_{k+1}}{(x_{r+1}^{n_{r+1}+1}, \ldots, x_k^{n_k+1})M_\fp:(x_{r+1}\ldots x_k)}.$$
Therefore, $\fa_k(M)R_\fp\subseteq \fa_{k-r}(M_\fp)$.
\end{proof}

In the rest of this section, we will use the results above to investigate the behavior of the depth of a module when passing to  localization. Note that this result has been known in some special cases (cf. Schenzel \cite{sch}). We start with a lemma.
\begin{lemma}\label{lem14}
Let $M$ be a finitely generated $R$-module and $x_1, \ldots, x_d$ be a system of parameters of $M$. Assume that $x_1, \ldots, x_d$ is a dd-sequence on $M$. For $k=0, 1, \ldots, d-1$ and $n\geq 5$ we have
$$\bigcap_{i=1}^{k+1}\Ann (0:x_i)_{M/(x_{i+1}^n, \ldots, x_{k+1}^n)M}\subseteq \sqrt{\fa_0(M)\ldots \fa_k(M)}.$$
\end{lemma}
\begin{proof}
The lemma is proved by induction on $k$. If $k=0$ then $0:_Mx_1=\bigcup_{i=1}^\infty0:_M(x_1^t, \ldots, x_d^t)=H^0_\fm(M)$. Thus $\Ann_R(0:_Mx_1)=\fa_0(M)$. Let $k>0$. The sequence $x_1, \ldots, \widehat{x_{k+1}}, \ldots, x_d$ is a dd-sequence on $M/x_{k+1}^nM$ by Lemma \ref{lem1}, $(ii)$. Hence, from the induction assumption and Corollary \ref{co13} we get
\[\begin{aligned}
&\bigcap_{i=1}^{k+1}\Ann (0:x_i)_{M/(x_{i+1}^n, \ldots, x_{k+1}^n)M}\\
&\subseteq \sqrt{\Ann_R (0:_Mx_{k+1})\fa_0(M/x_{k+1}^nM)\ldots \fa_{k-1}(M/x_{k+1}^nM)}\\
&=\sqrt{\Ann_R (0:_Mx_{k+1})\fa_0(M)\ldots \fa_{k-1}(M)\fa_1(M/0:_Mx_{k+1})\ldots \fa_{k-1}(M/0:_Mx_{k+1})}.
\end{aligned}\]
Note that $\Ann_R(0:_Mx_{k+1})\subseteq \fa_0(0:_Mx_{k+1})\cap \ldots \cap \fa_k(0:_Mx_{k+1})$. From the long exact sequence of local cohomology modules
\begin{multline*}
0\longrightarrow H^1_\fm(0:_Mx_{k+1})\longrightarrow H^1_\fm(M)\longrightarrow H^1_\fm(M/0:_Mx_{k+1})\longrightarrow \ldots \\
\ldots \longrightarrow H^k_\fm(0:_Mx_{k+1})\longrightarrow H^k_\fm(M)\longrightarrow H^k_\fm(M/0:_Mx_{k+1})\longrightarrow 0,
\end{multline*}
we imply that 
\begin{multline*}\sqrt{\Ann_R (0:_Mx_{k+1})\fa_1(M/0:_Mx_{k+1})\ldots \fa_{k-1}(M/0:_Mx_{k+1})}\\
\subseteq \prod_{i=0}^k\sqrt{\fa_i(0:_Mx_{k+1})\fa_i(M/0:_Mx_{k+1})}\subseteq \prod_{i=0}^k\sqrt{\fa_i(M)}.
\end{multline*}
Therefore, $$\bigcap_{i=1}^{k+1}\Ann (0:x_i)_{M/(x_{i+1}^n, \ldots, x_{k+1}^n)M}\subseteq \sqrt{\fa_0(M)\ldots \fa_k(M)}.$$
\end{proof}

\begin{proposition}\label{pro15}
Let $R$ be a Noetherian local ring and $M$ be a finitely generated $R$-module. Assume that $M$ has a p-standard system of parameters. Let $\fp\in \Supp(M)$ and $k\in \{0, 1, \ldots, d\}$. Then $\fp\supseteq \fa_0(M)\ldots \fa_k(M)$ if and only if $\depth M_\fp+\dim R/\fp\leq k$. Consequently, we have
\begin{multline*}
V(\fa_0(M)\ldots \fa_k(M))\setminus V(\fa_0(M)\ldots \fa_{k-1}(M))=\\
\{\fp\in\Supp(M):\depth M_\fp+\dim R/\fp=k\}.
\end{multline*}
\end{proposition}
\begin{proof} Set $r=\dim R/\fp$. Following Proposition \ref{pr5}, there is a p-standard system of parameters $x_1, \ldots, x_d$ of $M$ such that $x_{r+1}, \ldots, x_d\in\fp$. Note that $x_1, \ldots, x_d$ is a dd-sequence on $M$ by Lemma \ref{lem0}. Using Lemma \ref{lem14},  if $\fp\supseteq \fa_0(M)\ldots \fa_k(M)$ then $\fp\supseteq \Ann (0:x_i)_{M/(x_{i+1}^n, \ldots, x_{k+1}^n)M}$ for some $i\in \{1, \ldots, k+1\}, n\geq 5$.  Thus $x_i, \ldots, x_{k+1}\in \fp$ and $x_i, \ldots, x_{k+1}$ is not a regular sequence on $M_\fp$. Proposition \ref{pr5} implies $\depth M_\fp+\dim R/\fp\leq k$. 

Conversely, assume that $\fp\not\supseteq \fa_0(M)\ldots \fa_{k-1}(M)$, that is, $\fp\not \supseteq \fa_i(M)$ for $i=0, 1, \ldots, k-1$. From Corollary \ref{co6}, $R_\fp=\fa_i(M)R_\fp\subseteq \fa_{i-r}(M_\fp)$, hence $H^{i-r}_{\fp R_\fp}(M_\fp)=0$ for all $i<k$. In other words, $\depth M_\fp +\dim R/\fp\geq k$.
\end{proof}
The following corollary  generalizes a result of  Cuong \cite[Chapter 1, Proposition 3.3]{ntc1}  and is a direct consequence of Proposition \ref{pro15}.
\begin{corollary}\label{co16}
Assume that $M$ admits a p-standard system of parameters. Then
$$V(\fa(M))=\{\fp\in \Supp(M): \depth M_\fp+\dim R/\fp<d\}.$$
In particular, if $\fp\in \Supp(M)$ and $\fp\not\supseteq \fa(M)$ then $M_\fp$ is Cohen-Macaulay of dimension $d-\dim R/\fp$.
\end{corollary}
\begin{corollary}\label{co17}
Assume that $M$ admits a p-standard system of parameters.  For each $k\geq 0$, the set
$$\{\fp\in \Supp(M): \depth M_\fp+\dim R/\fp<k\},$$
is a closed subset of $\Supp(M)$ of dimension $\dim R/\fa_0(M)\ldots \fa_k(M)$.
\end{corollary}

Following Proposition \ref{pro15}, if a finitely generated $R$-module $M$ admits a p-standard system of parameters then 
$$\{\fp\in \Supp(M): \depth M_\fp+\dim R/\fp<k+1\}=V(\fa_0(M)\ldots \fa_k(M))$$
which is a closed subset in $\Spec(R)$. Hence, $$k\geq \max_\fp\{\depth M_\fp+\dim R/\fp<k+1\}\geq \dim R/\fa_0(M)\ldots\fa_k(M).$$ The next example shows that both inequalities are strict in general.
\begin{example}
Let $R=k[[X, Y, Z, T, W, U, V]]$ be the ring of formal power series with coefficients in a field $k$. Put $M=R/(Y, Z, T)\cap(W, U, V)$. It can be proved that $\dim M=4$ and $x_1=X, x_2=Y+W, x_3=Z+U, x_4=T+V$ is a system of parameters of $M$ and 
$$\ell(M/(x_1^{n_1}, x_2^{n_2}, x_3^{n_3}, x_4^{n_4})M)=2n_1n_2n_3n_4+n_1,$$
for all $n_1, n_2, n_3, n_4>0$. So $x_1, x_2^2, x_3^3, x_4^4$ is a p-standard system of parameters of $M$ and $\dim R/\fa(M)=1$ by Lemma \ref{lem0} and \cite[Corollary 3.9]{cc}. Moreover, $\nCM(M)=\{\fp, \fm\}$ where $\fp=(Y, Z, T, W, U, V)$. We have $\depth M_\fp=1$ and $\depth M_\fp+\dim R/\fp=\depth M=2$. Therefore,
$$3>\max_\fp\{\depth M_\fp+\dim R/\fp<4\}>\dim R/\fa(M).$$
\end{example}


\section{Arithmetic Macaulayfication of module}
\label{section4}


Let $M$ be a finitely generated $R$-module. An arithmetic Macaulayfication of $M$ is a Cohen-Macaulay Rees module $\cR(M, I):=\oplus_{n\geq 0}I^nM$ for some ideal $I$ with $\htt(I+\Ann_R(M))/\Ann_R(M)>0$. The existence of arithmetic Macaulayfication of ring has been studied by Kawasaki \cite{tk2, tk3}. In this section we will investigate the existence of arithmetic Macaulayfication of module and prove Theorem \ref{main1a}. p-Standard system of parameters plays a central role in this investigation.

The following proposition is a generalization of Theorem 3.2 of Zhou \cite{zh}.

\begin{proposition}\label{zh} Let $R$ be a Noetherian local ring and $M$ be a finitely generated $R$-module. Then $\dim R/\fa(M)<\dim M$ if and only if $\dim R/\fa(R/\fp)<\dim R/\fp$ for all primes $\fp\in \Ass(M)$ with $\dim R/\fp=\dim M$.
\end{proposition}
\begin{proof}
Let $U_M(0)$ be the biggest submodule of $M$ with $\dim U_M(0)<\dim M$. Since $M$ is Noetherian, $U_M(0)$ exists uniquely and 
$$\Ass M/U_M(0)=\{\fp\in \Ass(M): \dim R/\fp=\dim M\},$$
(see \cite[Remark 2.3]{cc1}). In particular, $M/U_M(0)$ is equidimensional. From Lemma \ref{am}, we have $\fa(M)\subseteq \Ann_R U_M(0)$. Combining this with the exact sequence of local cohomology modules 
$$\ldots \longrightarrow H_\fm^i(U_M(0))\longrightarrow H^i_\fm(M)\longrightarrow H_\fm^i(M/U_M(0)) \longrightarrow H_\fm^{i+1}(U_M(0))\longrightarrow \ldots $$
we get the inclusions
$$\fa(M/U_M(0))\subseteq \sqrt{\fa(M)}\subseteq \sqrt{\fa(M/U_M(0)) + \Ann_RU_M(0)}.$$
Hence, $\dim R/\fa(M)<\dim M$ if and only if $\dim R/\fa(M/U_M(0))<\dim M$. This reduces the proof to the case $M$ is equidimensional.

Now assume $M$ is equidimensional. A uniform local cohomology element of $M$ is an element 
$$x\in \fa_0(M)\cap\fa_1(M)\cap\ldots\cap \fa_{d-1}(M), \ d=\dim M,$$
which is not contained in any minimal associated prime ideal of $M$. Since $M$ is equidimensional, $x\in \fa_0(M)\cap\fa_1(M)\cap\ldots\cap\fa_{d-1}(M)$ is a uniform local cohomology element if and only if it is a parameter element. Hence, $M$ has a uniform local cohomology element if and only if $\dim R/\fa(M)<\dim M$. In \cite[Theorem 3.2]{zh} Zhou proves that if $S$ is an equidimensional local ring then $S$ has a uniform local cohomology element if and only if so does $S/\fp$ for all primes $\fp\in\Ass S$, $\dim S/\fp=\dim S$. The proof in fact works well if we replace the ring $S$ by an equidimensional finitely generated $S$-module. In particular, in our case $M$ has a uniform local cohomology element if and if so does $R/\fp$ for any $\fp\in \Ass M$. This completes the proof.
\end{proof}

The first theorem of this section relates the existence of p-standard system of parameters on modules and on the ground ring. It is in fact a starting point of the works in this paper. Recall that a uniform local cohomology annihilator of $M$ is defined by Huneke \cite{hu1} to be an element which kills all local cohomology modules $H^i_\fm(M)$ with $i\not=\dim M$ and which is not contained in any minimal associated prime ideal of the module $M$. 

\begin{theorem}\label{main1} Let $R$ be a Noetherian local ring and $M$ be a finitely generated $R$-module. The following statements are equivalent:
\begin{enumerate}[(a)]
\item $M$ admits a p-standard system of parameters.

\item Any quotient domain of $R/\Ann_RM$ has a uniform local cohomology annihilator. 

\item $R/\Ann_RM$ admits a p-standard system of parameters. 

\item Any finitely generated $R$-module $N$ with $\Supp(N)\subseteq \Supp(M)$ admits a p-standard system of parameters.
\end{enumerate}
\end{theorem}
\begin{proof}

\noindent{$(a)\Rightarrow (b)$}. Assume that $M$ admits a p-standard system of parameters. Let $S$ be a quotient of $R/\Ann_RM$ which is a domain. We have to show that $\fa(S)\not=0$. 

We take any prime ideal $\fp\in \Supp(M)$ and denote $S=R/\fp$. Set $d=\dim M$ and $r=\dim S$. By Proposition \ref{pr5}, there is a p-standard system of parameters $x_1, \ldots, x_d$ of $M$ such that $x_{r+1}, \ldots, x_d\in \fp$. Denote $N=M/(x_{r+1}, \ldots, x_d)M$. Then 
$\dim S=r=\dim N,$ and hence, $\fp\in \Ass N$. Note that $x_1, \ldots, x_r$ is a p-standard system of parameters of $N$ by definition. In particular, we have $\dim R/\fa(N)<\dim N$. Applying Proposition \ref{zh} to the module $N$ and the prime $\fp\in \Ass N$, we obtain $\dim S/\fa(S)<\dim S$, so $\fa(S)\not=0$.
This proves $(a)\Rightarrow (b)$.
\medskip

\noindent $(b)\Rightarrow (d)$. Let $N$ be any finitely generated $R$-module with $\Supp(N)\subseteq \Supp(M)$. 
Proposition \ref{zh} shows that $\dim R/\fa(N)<\dim N$. So there exists a parameter element $x_t\in \fa(N)$ of $N$, where $t=\dim N$. By induction on $t$, we obtain a system of parameters $x_1, \ldots, x_t$ of $N$ such that $x_i\in \fa(N/(x_{i+1}, \ldots, x_t)N)$ for $i=t, t-1, \ldots, 1$, that is, $x_1, \ldots, x_t$ is a p-standard system of parameters of $N$.
\medskip

\noindent $(d)\Rightarrow (a)$. Straightforward.
\medskip

\noindent $(b)\Leftrightarrow (c)$. This equivalence is a special case of $(a)\Leftrightarrow (b)$.
\end{proof}


\begin{corollary}\label{corollarythem}
Keep the assumptions as in Theorem \ref{main1}. A finitely generated $R$-module $M$ admits a p-standard system of parameters if and only if so does $R/\fp$ for all minimal associated prime ideals $\fp$ of $M$.
\end{corollary}
\begin{proof}
The necessary condition is immediate from Theorem \ref{main1}. For the converse, let $\fq\in \Supp(M)$ be any  prime ideal. Then there is an minimal associated prime ideal $\fp$ of $M$ such that $\fp\subseteq \fq$. Theorem \ref{main1} then implies that $R/\fq$ admits a p-standard system of parameters. Therefore, $M$ admits a p-standard system of parameters.
\end{proof}

 We now use Theorem \ref{main1} to prove the main result of this section on arithmetic Macaulayfication of module. Recall that a finitely generated $R$-module is unmixed if for all associated prime ideals $P$ of the $\fm$-adic completion $\hat M$, $\dim \hat R/P=\dim M$. This is clearly equivalent to saying that $R/\Ann_R(M)$ is unmixed. 

\begin{theorem}\label{macmodule}
Let $R$ be a Noetherian local ring and $M$ be a finitely generated $R$-modules. The following statements are equivalent:

\begin{enumerate}[(a)]
\item $M$ has an arithmetic Macaulayfication.

\item $M$ is unmixed and for all associated prime ideals $\fp\in \Ass(M)$, $R/\fp$ has an arithmetic Macaulayfication.

\item $M$ is unmixed and admits a p-standard system of parameters.
\end{enumerate}
\end{theorem}
\begin{proof} 
\noindent{$(a)\Rightarrow (c)$}. Replace $R$ by $R/\Ann_R(M)$, we can assume that $\Ann_R(M)=0$. Let $\cR(M, I)$ be an arithmetic Macaulayfication of $M$, where $I\subset R$ is an ideal with $\htt(I)>0$. 

We first show that $M$ is unmixed. Since $M$ is a faithful $R$-module, this is equivalent to showing that $R$ is unmixed. By flat base change, $\cR(\hat M, I\hat R)$ is a Cohen-Macaulay $\cR(\hat R, I\hat R)$-module. Note that the injective canonical homomorphism $R\hookrightarrow \Hom_R(M, M)$ derives an injective homomorphism $\hat R\hookrightarrow \Hom_{\hat R}(\hat M, \hat M)$ (flat base change). Thus $\hat M$ is a faithful $\hat R$-module. So  $\cR(\hat M, I\hat R)$ is faithful over $\cR(\hat R, I\hat R)$. The Cohen-Macaulayness of  $\cR(\hat M, I\hat R)$ implies that for any associated prime ideal $\mathcal P$ of $\cR(\hat R, I\hat R)$, 
$$\dim \cR(\hat R, I\hat R)/\mathcal P=\dim \cR(\hat R, I\hat R)=\dim \hat R+1.$$
So, $\dim \hat R/P=\dim \hat R$ for any associated prime ideal $P\in\Ass(\hat R)$. This shows that $M$ is unmixed.


Let $R^\prime, M^\prime$ respectively be the localizations of the Rees algebra $\cR(R, I)$ and the Rees module $\cR(M, I)$ at the maximal homogeneous ideal $\fm\oplus\cR(R, I)_+:=\fm\oplus I\oplus I^2\oplus \ldots$.
Note that $R$ is a quotient of the local ring $R^\prime$ and the $R$-module $M$ is a quotient of the Cohen-Macaulay $R^\prime$-module $M^\prime$. Since any regular system of parameters of $M^\prime$ is also p-standard, Theorem \ref{main1} implies that $M$ admits a p-standard system of parameters.  

\medskip

\noindent{$(c)\Rightarrow (a)$}. Let $\dim M=d$ and $x_1, \ldots, x_d\in R$ be a p-standard system of parameters of $M$. Since $M$ is unmixed, then $\dim R/\fp=\dim M$ for all $\fp\in\Ass(M)$, and thus $x_i$'s are regular elements of $M$. Equivalently, $0:_Mx_1=\ldots=0:_Mx_d=0$.

Denote $I_i=(x_i, \ldots, x_d)\subset R$ for $i=1, 2, \ldots, d$. We define the multi-graded Rees module of $M$ with respect to the ideals $I_1, \ldots, I_{d-1}$ by
$$\cR(M, I_1, \ldots, I_{d-1})=\bigoplus_{n_1,\ldots, n_{d-1}\geq 0}I_1^{n_1}\ldots I_{d-1}^{n_{d-1}}MT_1^{n_1}\ldots T_{d-1}^{n_{d-1}},$$
where $T_1, \ldots, T_{d-1}$ are indeterminates. Arguing as in the proof of Corollary $4.5$ of Kawasaki \cite{tk2}, $\cR(M, I_1, \ldots, I_{d-1})$ is a Cohen-Macaulay module over the corresponding multi-graded Rees algebra $\cR(R, I_1, \ldots, I_{d-1})$. Then following Theorem $4.5$ of Chan-Cumming-T\`ai \cite{cct}, the Rees module $\cR(M, I_1\ldots I_{d-1})$ is Cohen-Macaulay which is an arithmetic Macaulayfication of $M$.

\medskip
\noindent{$(b)\Rightarrow (c)$}. If for all $\fp\in \Ass(M)$, $R/\fp$ has an arithmetic Macaulayfication then by the equivalence $(a)\Leftrightarrow (c)$, $R/\fp$ has a p-standard system of parameters. The conclusion then follows by Corollary \ref{corollarythem}.

\medskip

\noindent{$(c)\Rightarrow (b)$}. Assume $M$ is unmixed and admits a p-standard system of parameters. Let $\fp$ be an associated prime ideal of $\Ass(M)$. Then $R/\fp$ is unmixed following the embedding $\hat R/\fp\hat R\hookrightarrow \hat M$. Theorem \ref{main1} tells us that $R/\fp$ has a p-standard system of parameters. Then the equivalence $(a)\Leftrightarrow (c)$ implies $R/\fp$ has an arithmetic Macaulayfication.
\end{proof}

Denote by $\cQ(R)$ the set of $\fp\in \Spec(R)$ such that $R/\fp$ has an arithmetic Macaulayfication. $\cQ(R)$ contains the maximal ideal $\fm$ and all prime ideals $\fp\in \Spec(R)$ with $\dim R/\fp=1$. Indeed, obviously $R/\fp$ is a Cohen-Macaulay ring. So Theorem \ref{macmodule} implies that $R/\fp$ has an arithmetic Macaulayfication since any regular system of parameters is also p-standard.
\begin{corollary}\label{cor42}
 Let $R$ be a Noetherian local ring. A prime ideal $\fp$ is in $\cQ(R)$ if and only if $R/\fp$ admits a p-standard system of parameters. The set $\cQ(R)$ is stable under specialization. Moreover, a finitely generated $R$-module $M$ admits a p-standard systems of parameters if and only if $\Supp(M)\subseteq \cQ(R)$.
\end{corollary}
\begin{proof}
Following Theorem \ref{macmodule}, for a prime ideal $\fp\in \Spec(R)$, $R/\fp$ has an arithmetic Macaulayfication if and only if it admits a p-standard system of parameters. Therefore the first assertion is clear.

Let $\fp\in\cQ(R)$, then $R/\fp$ admits a p-standard system of parameters. Theorem \ref{main1} tells us that for any $\fq\in \Spec(R)$ with $\fp\subseteq \fq$, $R/\fq$ also admits a p-standard system of parameters, hence $\fq\in \cQ(R)$. So $\cQ(R)$ is stable under specialization.

The last assertion is a direct consequence of Theorem \ref{main1}.
\end{proof}

By this corollary, the subset $\cQ(R)$ of $\Spec(R)$ is stable under specialization, so it is natural to ask if it is closed in the Zariski topology or not. The answer to this question is negative in general. If we take the two-dimensional local domain given by Ferrand-Raynaud \cite[Proposition 3.3]{fr} again, denoted $R$, then this domain is not a quotient of a Cohen-Macaulay local ring as it is not universally catenary. Hence
$$\cQ(R)=\{\fp\in \Spec(R): \dim R/\fp\leq 1\},$$
which is not a closed subset of $\Spec(R)$.

\begin{corollary}
Let $R$ be a Noetherian local ring ($R$ might have no p-standard systems of parameters). Let $\mathcal{S}(R)$ be the full subcategory of $\mathcal{M}od(R)$ whose objects are those finitely generated $R$-modules which admit a p-standard system of parameters. For examples, $\mathcal{S}(R)$ contains all $0$- and $1$-dimensional finitely generated $R$-modules. Then $\mathcal{S}(R)$ is an abelian subcategory of $\mathcal{M}od(R)$ which is also a Serre subcategory, that is, for any short exact sequence of $R$-modules in $\mathcal{M}od(R)$,
$$0\longrightarrow N \longrightarrow M\longrightarrow L\longrightarrow 0,$$
$M\in \mathcal S(R)$ if and only if $N, L\in \mathcal S(R)$.

Moreover, $\mathcal S(R)$ is closed under the tensor operator, taking subquotients of modules and applying $\Hom_R(-, R)$.
\end{corollary}
\begin{proof}
The corollary is a direct consequence of Theorem \ref{main1} and Corollary \ref{cor42}. In particular, if
$$0\longrightarrow N \longrightarrow M\longrightarrow L\longrightarrow 0$$
is a short exact sequence of $R$-modules then $\Supp(M)=\Supp(N) \cup \Supp(L)$. Thus, by Corollary \ref{cor42}, both $M\in \mathcal S(R)$ and $N, L\in \mathcal S(R)$ are equivalent to $\Supp(M)\subseteq \cQ(R)$. 

Furthermore, if $M, N\in \mathcal S(R)$ then $M\otimes_RN\in \mathcal S(R)$ since $\Supp(M\otimes_RN)=\Supp(M)\cap \Supp(N)\subseteq \cQ(R)$ (cf. Bourbaki \cite[Chapitre II, Proposition 18]{bb}).
\end{proof}


\section{Quotients of Cohen-Macaulay local rings}
\label{section5}

In this section we apply results in previous sections to explore a relationship between p-standard system of parameters and quotient of Cohen-Macaulay local ring. This could be viewed as a generalization formulation of the well-known fact that Cohen-Macaulayness is characterized by certain local cohomology vanishing.

We first recall the following characterization of quotients of Cohen-Macaulay rings by Kawasaki \cite[Corollary 1.2]{tk2}.

\begin{theorem}[Kawasaki]\label{kawasakiCM}
A Noetherian local ring is a quotient of a Cohen-Macaulay local ring if and only if it is universally catenary and all its formal fibers are Cohen-Macaulay.
\end{theorem}

This theorem is very interesting as it relates a property on a ring to a property on the formal fibers of the ring, a classical topic in commutative algebra (see \cite{gr}). It will be used effectively in our proof of Theorem \ref{main3a}. The first two statements of Theorem \ref{main3a} being equivalent is proved in the following theorem.

\begin{theorem}\label{main11}
A Noetherian local ring is a quotient of a Cohen-Macaulay local ring if and only if it has a p-standard system of parameters. 
\end{theorem}

\begin{proof}
\noindent {\it Necessary condition:} Assume that $R$ is a quotient of a Cohen-Macaulay local ring $S$. Note that any regular system of parameters of $S$ is also a p-standard system of parameters. So any finitely generated $S$-module admits a p-standard system of parameters by Theorem \ref{main1}(d). In particular, the quotient $R$ admits a p-standard system of parameters.

\medskip
\noindent {\it Sufficient condition:} Assume that $R$ admits a p-standard system of parameters. Following Theorem \ref{kawasakiCM}, in order to prove the sufficient condition, we will show that the ring is universally catenary and all its formal fibers are Cohen-Macaulay.

a. Let $\fp\in \Spec(R)$ and set $S=R/\fp$. We have $\fa(S)\not=0$ by Theorem \ref{main1}$(b)$. Note that any non-zero divisor of $S$ is also a non-zero divisor of $\hat S$ since the completion is faithfully flat, then any non-zero element of $\fa(S)$ is not a zero divisor of $\hat S$. Note also that $\fa(S)\hat S\subseteq \fa(\hat S)$ by the isomorphism $H^i_\fm(S)\otimes_S\hat S\simeq H^i_{\fm\hat S}(\hat S)$ for all $i\geq 0$. Hence  there are elements in $\fa(\hat S)$ which are not zero divisors of $\hat S$. Combining this simple observation with Lemma \ref{am}, we get that $\hat S$ is unmixed. Therefore $R$ is universally catenary by Grothendieck \cite[Proposition 7.1.11]{gr} (see also Matsumura \cite[Theorem 31.7]{ma}). 

b. Let $\fq\in \Spec (R)$ be any prime ideal. We need to show that the fiber ring $\hat R\otimes_Rk(\fq)$ is Cohen-Macaulay where $k(\fq)=R_\fq/\fq R_\fq$. Since the canonical map $R/\fq \longrightarrow \hat R/\fq\hat R$ provides the completion of $R/\fq$ and $R/\fq$ also admits a p-standard system of parameters by Theorem \ref{main1}, we reduce to the case $R$ is a domain and $\fq=0$. 

Let $\hat\fp\in \Spec(\hat R)$ such that $\hat \fp\cap R=0$. Denote by $K$ the field of fractions of $R$. Combining the fact $\fa(R)\hat R\subseteq \fa(\hat R)$ as prove in (a) with the assumption $\dim R/\fa(R)<\dim R$, we imply that $\dim \hat R/\fa(\hat R)<\dim R$. Therefore, $\hat \fp\not\supseteq \fa(\hat R)$. By the structure theorem for Noetherian complete local rings \cite[Theorem 29.4]{ma}, the $\fm$-adic completion $\hat R$ is a quotient of a regular local ring, in particular, a Cohen-Macaulay local ring. So by the necessary condition, $\hat R$ admits a p-standard system of parameters. It implies from Corollary \ref{co16} that $\hat R_{\hat \fp}$ is Cohen-Macaulay. Finally, for any $\fp^\prime \in \Spec(\hat R\otimes_RK)$, set $\hat\fp=\fp^\prime\cap \hat R$, then $\hat \fp\cap R=0$.  Since $(\hat R\otimes_RK)_{\fp^\prime}\simeq \hat R_{\hat \fp}\otimes_K K\simeq \hat R_{\hat \fp}$ is Cohen-Macaulay, $\hat R\otimes_R K$ is Cohen-Macaulay.
\end{proof}

The module version  of Theorem \ref{main11} is as follows.

\begin{corollary} Let $M$ be a finitely generated $R$-module. Then $M$ admits a p-standard system of parameters if and only if $R/\Ann_R(M)$ is a quotient of a local ring $R^\prime$ and $M$ is a quotient of a finitely generated Cohen-Macaulay $R^\prime$-module.
\end{corollary}
\begin{proof} Since $M$ is finitely generated, there is an surjective homomorphism $\varphi: (R/\Ann_R(M))^r\rightarrow M$. Replacing $R$ by $R/\Ann_R(M)$ if it is necessary, we assume that $\Ann_R(M)=0$.

If $M$ admits a p-standard system of parameters then $R$ also admits a p-standard system of parameters by Theorem \ref{main1}. Theorem \ref{main11} then implies that $R$ is a quotient of a Cohen-Macaulay local ring $R^\prime$. The composition $(R^\prime)^r\rightarrow R^r\rightarrow M$ then implies that $M$ is a quotient of a Cohen-Macaulay $R^\prime$-module.

Conversely, assume $R$ is a quotient of a local ring $R^\prime$ and $M$ is a quotient of a finitely generated Cohen-Macaulay $R^\prime$-module. Here $R^\prime$ is a Noetherian local ring and $R$ is a quotient of $R^\prime$. Then it is clear from Theorem \ref{main1} that $M$ admits a p-standard system of parameters over $R$.
\end{proof}

Theorem \ref{main11} has several interesting consequences. The first corollary is implied from Theorem \ref{macmodule} and Theorem \ref{main11}.

\begin{corollary}
Let $M$ be a finitely generated and faithful module over a Noetherian local ring $R$. Then an arithmetic Macaulayfication of $M$ exists if and only if $M$ is unmixed and $R$ is a quotient of a Cohen-Macaulay local ring.
\end{corollary}

The locus of non-Cohen-Macaulay points of the spectrum of a Noetherian ring (also, of the support of a finitely generated module on it) keeps many information about the structure of the ring (respectively, the module). In the next corollary, we will give a relation between this set and the annihilator ideals of certain local cohomology modules of module who admit a p-standard system of parameters. This is a remarkable generalization of a result of Schenzel \cite[2.4.6]{sch}. Recall first that for any finitely generated $R$-module $M$, following \cite[Remark 2.3]{cc1} there is a filtration of submodules of $M$,
$$D_0\subset \ldots \subset D_{t-1}\subset D_t=M,$$
where each $D_{i-1}$ is the biggest submodule of $D_i$ with $\dim D_{i-1}<\dim D_i$, $i=t, \ldots, 1$ and $D_0=H^0_\fm(M)$-the biggest submodule of finite length of $M$. This filtration always exists uniquely and is called the dimension filtration of $M$.
\begin{corollary} Suppose $R$ is a quotient of a Cohen-Macaulay local ring. Denote $\nCM(M):=\{\fp\in \Supp(M): M_\fp \text{ is not Cohen-Macaulay}\}$. We have $$\nCM(M)=\bigcup_{i=0}^t\big(V(\fa(D_i))\cap\Supp(D_i/D_{i-1})\big).$$
Consequently, $\nCM(M)$ is a closed subset of $\Supp(M)$ (with the Zariski topology). 
\end{corollary}
\begin{proof}
Let $\fp\in \Supp(M)$ and $i=0, \ldots, t$ such that $\fp\in \Supp(D_i/D_{i-1})$ but $\fp\not\in \Supp(D_{j}/D_{j-1})$ for $j=i+1, \ldots, t$. Such an $i$ exists since $\Supp(M)=\cup_{j=0}^t\Supp(D_j/D_{j-1})$ where we put $D_{-1}=0$. It is obvious that $M_\fp=(D_i)_\fp$. Note that $\Ass(D_i/D_{i-1})=\{\fq\in \Ass(D_i): \dim R/\fq=\dim D_i\}$ by the choice of the filtration (see \cite[Remark 2.3]{cc1}) and the catenaricity of $R$, this shows that $\dim (D_i)_\fp+\dim R/\fp=\dim D_i$. So $(D_i)_\fp$ is Cohen-Macaulay if and only if $\depth (D_i)_\fp+\dim R/\fp=\dim D_i$. Using Theorem \ref{main1} and Corollary \ref{co16}, this is equivalent to saying that $\fp\not\supseteq \fa(D_i)$. Therefore, 
$$\nCM(M)=\bigcup_{i=0}^t(V(\fa(D_i))\cap\Supp(D_i/D_{i-1})),$$
which is closed in $\Supp(M)$.
\end{proof}

A Noetherian local ring $R$ is called a generalized Cohen-Macaulay ring (see \cite{cst}) if the local cohomology modules $H^i_\fm(R)$'s are finitely generated $R$-modules for $i=0, 1, \ldots, \dim R-1$, or equivalently, if $\dim R/\fa(R)\leq 0$. Hence any system of parameters of a generalized Cohen-Macaulay local ring with a big enough exponent is a p-standard system of parameters. 
From Theorem \ref{main11} we have the following corollary.

\begin{corollary}\label{p1}
Let $R$ be a Noetherian local ring. Assume that $\dim R/\fa(R)\leq 1$. Then $R$ is a quotient of a Cohen-Macaulay local ring. In particular, any generalized Cohen-Macaulay ring is a quotient of a Cohen-Macaulay local ring.
\end{corollary}
\begin{proof}
Since $\dim R/\fa(R)\leq 1$, one can choose $x_2, \ldots, x_n\in \fa(R)$ which is a part of a
system of parameters of $R$ and choose $x_1\in \Ann H^0_\fm(R/(x_2, \ldots,
x_n))$ which is a parameter element of $R/(x_2, \ldots, x_n)$. Then $x_1^r,  \ldots,
x_n^r$ is a p-standard system of parameters of $R$ for $r\geq n!$ by Lemma \ref{lem3.1}. The corollary then follows from Theorem \ref{main11}.
\end{proof}

We say that a Noetherian local ring $R$ has a small Cohen-Macaulay module if there is a finitely generated Cohen-Macaulay $R$-module $M$ such that $\dim M=\dim R$. Note that all rings of dimension $1$ have a small Cohen-Macaulay module, namely, $R/H^0_\fm(R)$. However, as being indicated in the following corollary, it is not the case of rings of higher dimension (see also Grothendieck \cite[Proposition 6.3.8]{gr}).
\begin{corollary}
Let $R$ be a Noetherian local ring. If $R$ has a small Cohen-Macaulay module $M$ such that $\Supp(M)=\Spec(R)$ then $R$ is a quotient of a Cohen-Macaulay local ring. In particular, for any $d\geq 2$ there is a Noetherian local domain of dimension $d$ which does not have any small Cohen-Macaulay module.
\end{corollary}
\begin{proof}
Since any regular system of parameters is also a p-standard system of parameters, Theorem \ref{main11} implies that $R$ is a quotient of a Cohen-Macaulay local ring. 

Let $A$ be the two-dimensional local domain constructed by Ferrand and Raynaud  in \cite[Proposition 3.3]{fr}. It was proved in that paper that $\hat A$ has an embedding associated prime, hence $A$ is not universally catenary and it does not admit any p-standard system of parameters. For each $n\geq 0$, let $R_n=A[[X_1, \ldots, X_n]]$ be the ring of formal power series with coefficients in $A$. Then $\dim R_n=n+2$. Note that $A$ is a quotient of $R_n$. By Theorem \ref{main1}, $R_n$ admits no p-standard systems of parameters, hence, small Cohen-Macaulay modules over $R_n$ do not exist.
\end{proof}

Assume that $(R, \fm)$ is a Noetherian local domain which has a small Cohen-Macaulay module $M$. Then there is a short exact sequence
$$0\longrightarrow R \longrightarrow M \longrightarrow M^\prime \longrightarrow 0,$$
which derives a long exact sequence of local cohomology modules
$$H^0_\fm(M)\longrightarrow H^0_\fm(M^\prime)\longrightarrow H^1_\fm(R)\longrightarrow H^1_\fm(M)\longrightarrow H^1_\fm(M^\prime).$$
Since $H^0_\fm(M)=0$ and $H^1_\fm(M)=0$, $H^1_\fm(R)\simeq H^0_\fm(M^\prime)$ is of finite length. The next corollary says that a similar property also holds for quotient domains of Cohen-Macaulay local rings. It should be mentioned here that this corollary was proved by P. Schenzel \cite{sch} for local rings admitting a dualizing complex.
\begin{corollary}\label{assprime}
Let $(R, \fm)$ be a quotient of a Cohen-Macaulay local ring and $M$ be a finitely generated $R$-module. Then $\dim R/\fa_i(M)\leq i$ for all $i$ and the equality holds if and only if there is an associated prime $\fp\in \Ass(M)$ such that $\dim R/\fp=i$. 
\end{corollary}

\begin{proof}
Put $d=\dim M$. It is obvious from the theory of local cohomology  that $\dim R/\fa_d(M)=d$. By Theorems \ref{main1} and \ref{main11},  $M$ admits a p-standard system of parameters. Applying Proposition \ref{pro15} we imply that $\dim R/\fa_i(M)\leq i$ for $i=0, \ldots, d-1$. The equality holds if and only if there is a $\fp\in \Supp(M)$ such that $\dim R/\fp=i$ and $\depth(M_\fp)=0$, that is, $\fp \in \Ass(M)$.
\end{proof}


\section{Macaulayfication of local rings}
\label{section6}

This last section is devoted to discussing on Macaulayfication of local rings. Firstly, combining Theorem \ref{main1} and Theorem \ref{main11} with the works of  Grothendieck \cite{gr}, Faltings \cite{gf1} and Kawasaki \cite{tk1}, we obtain the following theorem which proves the rest of Theorem \ref{main3a}.

\begin{theorem}\label{main12}
Let $R$ be a Noetherian local ring. The following statements are equivalent:
\begin{enumerate}[(a)]
\item $R$ is a quotient of a Cohen-Macaulay local ring.

\item There is a finitely generated $R$-module $M$ with $\Supp(M)=\Spec(R)$ which admits a p-standard system of parameters.

\item $R$ is universally catenary and for each quotient $S$ of $R$,  $\Spec(S)$ has a Macaulayfication.

\item All essentially of finite type $R$-algebras verify Faltings' Annihilator Theorem. That means, if we let $A$ be an  essentially of finite type $R$-algebra and $N$ be a finitely generated $A$-module, then for any pair of ideals $\fb\subseteq \fa$ of $A$, we have
$$\inf \{i: \fb\not\subseteq \sqrt{\Ann_A H^i_\fa(N)}\}=\inf\{\depth N_\fp+\htt(\fa+\fp)/\fp: \fp\notin V(\fb)\}.$$

\item Any of the above statements holds for all essentially of finite type $R$-algebras.
\end{enumerate}
\end{theorem}


\begin{proof}
\noindent $(a) \Leftrightarrow (b)$. Let $M$ be a finitely generated $R$-module such that $\Supp(M)=\Spec(R)$. Theorem \ref{main1}$((a)\Leftrightarrow(d))$, first apply to $M$ and then apply to $R$, shows that $M$ admits a p-standard system of parameters if and only if so does $R$. By Theorem \ref{main11}, $R$ admits a p-standard system of parameters if and only if $R$ is a quotient of a Cohen-Macaulay local ring. This proves $(a)\Leftrightarrow (b)$.

\medskip
\noindent $(a) \Rightarrow (c)$. Assume $R$ is a quotient of a Cohen-Macaulay local ring. Then it is universally catenary. 
 Let $S$ be a quotient of $R$. We write $S=R/\cap_{i=1}^r\fq_i$ where $\fq_1, \ldots, \fq_r$ are primary ideals of $R$. For each $j=0, 1, \ldots, \dim S$, put
$$S_j=R/\cap_{\dim R/\fq_i=j} \fq_i.$$
So each $S_j$ is either zero or equidimensional of dimension $j$. We have a natural morphism
$$f: \bigsqcup_{j=0}^{\dim S}\Spec(S_j) \rightarrow \Spec(S),$$ where $\sqcup$ denotes the disjoint union.
This morphism is birational and proper. Indeed, since $f$ is finite, it is proper. We denote by $\CM(S)$ the Cohen-Macaulay locus of $S$, that is, those prime ideal $\fp$ such that $S_\fp$ is Cohen-Macaulay. If $\fp\in \CM(S)$ then $S_\fp$ is equidimensional. Combining this with the assumption on catenaricity of $R$, we get that $\fp\in \CM(S_j)$ for some $j$ and $\fp\not\in \CM(S_i)$ if $i\not=j$. So $f$ induces an isomorphism
$$\bigsqcup_{j=0}^{\dim S}\CM(S_j) \stackrel{\simeq}{\rightarrow} \CM(S).$$ It in particular shows that $f$ is birational.  If $S_j$ has a Macaulayfication  $X_j\rightarrow \Spec(S_j)$  for all $j$. Then the composition morphism 
$$\bigsqcup_{j=0}^{\dim S}X_j \rightarrow \bigsqcup_{j=0}^{\dim S}\Spec(S_j) \stackrel{\simeq}{\rightarrow} \Spec(S),$$
is a Macaulayfication of $\Spec(S)$. Therefore it remains to prove that for any equidimentional quotient $S$ of $R$, $\Spec(S)$ has a Macaulayfication.

Let $S$ be an equidimensional quotient of $R$. By Theorem \ref{main11}, $S$ admits a p-standard system of parameters. Corollary 4.2 of \cite{tk1} concludes that $\Spec(S)$ has a Macaulayfication. Here it is worth noting that in \cite[Corollary 4.2]{tk1}, the ring is assumed to possess a dualizing complex, but this is used only to guarantee that the ring has a p-standard system of parameters. So in our situation, we do not need this assumption.


\medskip
\noindent $(c) \Rightarrow (a)$. By Theorem \ref{kawasakiCM} it suffices to show that the formal fibers of $R$ are Cohen-Macaulay.  Moreover, replace $R$ by  $R/\frak p$ for any prime ideal $\fp\in \Spec R$, we  need only to show that the generic  formal fibers of $R$   are Cohen-Macaulay.

Firstly note that  if $X$ is an $\hat R$-scheme of finite type then the non-Cohen-Macaulay locus 
$$nCM(X)=\{x\in X: \mathcal O_{X, x} \text{ is not Cohen-Macaulay} \}$$
 is a closed subset of  $X$. This fact should be known by experts, but for the sake of completeness we give a proof for this fact. It suffices to prove this for an affine scheme, hence we assume $X=\Spec A$ for a finitely generated $\hat R$-algebra $A$. For each $\fp\in\Spec A$, we know by Grothendieck \cite[Corollare 6.12.8]{gr} that the singular locus $\Sing(A/\fp)$ is closed. Then $\nCM(A/\fp)$ is a subset of a closed subset of $\Spec A/\fp$ and by Nagata criterion (see Matsumura \cite[Theorem 24.5]{ma}), $\nCM(X)$ is a closed subset of $X$. 


Next, let $\pi: X\rightarrow \Spec R$ be a Macaulayfication. Denote $X^\prime=X\times_{\Spec R}\Spec\hat R$. We have a commutative diagram
\[\begin{CD}
X^\prime  @>\pi^\prime>> \Spec \hat R\\
@V\varphi^\prime VV  @V\varphi VV\\
X @>\pi>> \Spec R.
\end{CD}\]
We need to show that $X^\prime$ is a Macaulayfication of $\Spec \hat R$. Indeed, since $\pi$ is birational and of finite type, there is a an open dense subset $U\subseteq \Spec R$ such that the restriction $\pi^{-1}(U)\rightarrow U$ is an isomorphism and $\pi^{-1}(U)$ is dense in $X$. So $(\varphi^\prime)^{-1}(\pi^{-1}(U))\rightarrow \varphi^{-1}(U)$ is an isomorphism.
 Since $\pi^\prime$ is of finite type, $X^\prime$ is Noetherian. Moreover, the subset $\CM(X^\prime):=X^\prime\setminus \nCM(X^\prime)$ is open and dense in $X^\prime$ as we have proved. Let $\frak m$ and $\hat {\frak m}$ be the maximal ideals of $R$ and $\hat R$ respectively. Take a closed point $x^\prime\in (\pi^\prime)^{-1}(\hat {\frak m})$ of $X^\prime$. Put $x=\varphi^\prime(x^\prime)$ and $B=\cO_{X, x}$ with the maximal ideal $\fm_x$. Let $B^\prime=B\otimes_R\hat R$. Since $k_R(\fm) \simeq k_{\hat R}(\hat \fm)$, $(\pi^\prime)^{-1}(\hat {\frak m})\simeq \pi^{-1}(\frak m)$ where the map is induced from $\varphi^\prime$. Then there is a unique maximal ideal $\fn$ of $B^\prime$ over $\fm_x$, $\hat\fm$ such that $\cO_{X^\prime, x^\prime}\simeq B^\prime_\fn$. Lemma 7.9.3.1 of \cite{gr} gives us $\hat B\simeq \widehat{B^\prime_\fn}\simeq \widehat{\cO_{X^\prime, x^\prime}}$. So $\cO_{X^\prime, x^\prime}$ is Cohen-Macaulay since $B$ is Cohen-Macaulay. What we have proved is that $(\pi^\prime)^{-1}(\hat {\frak m})\subseteq \CM(X^\prime)$. Thus there is an open neighborhood $V\subseteq \Spec\hat R$ of $\hat \fm$  such that $(\pi^\prime)^{-1}(V)\subseteq \CM(X^\prime)$. But $\hat \fm$ is the unique closed point of $\Spec\hat R$, hence $V=\Spec\hat R$. Therefore, $X^\prime=\pi^{-1}(V^\prime)\subseteq \CM(X^\prime)$ is Cohen-Macaulay.

Now,  let $\fp$ be a minimal prime ideal of $R$ and let $P\in \Spec \hat R$ be a closed point in the fiber of $\varphi$ at $\fp$, that is, $P\cap R=\fp$. Since $\varphi^{-1}(U)$ is open and dense in $\Spec(\hat R)$, $P\in \varphi^{-1}(U)$. Hence, $(k_R(\fp)\otimes_R\hat R)_P\simeq \hat R_P$ is Cohen-Macaulay, where $k_R(\fp):=R_\fp/\fp R_\fp$. So that the fiber of $\varphi$ at $\fp$ is Cohen-Macaulay.

\medskip
\noindent $(a) \Leftrightarrow (d)$. This equivalence was proved by Faltings in \cite[Satz 4]{gf1}.

\medskip
\noindent $(a)\Leftrightarrow (e)$. The conclusion is clear since any polynomial ring of finitely many indeterminates over a Cohen-Macaulay ring is also Cohen-Macaulay.
\end{proof}

So if a Noetherian local ring is a quotient of a Cohen-Macaulay local ring, the spectrum of any quotient of the ring has a Macaulayfication. It is interesting to know whether the converse is true or not. The second part of this section is devoted to a discussion on this question.

We have seen in the proof of the equivalence $(a)\Leftrightarrow (c)$ in Theorem \ref{main12} that if the spectrum of any quotient of a Noetherian local ring has a Macaulayfication then all formal fibers of the ring are Cohen-Macaulay. Wishing to find a positive answer to the question above, by Theorem \ref{kawasakiCM} we must therefore prove that in addition, the ring is universally catenary. In the sequence, we will show that this is reduced to proving the catenaricity of the ring which is formulated in the following conjecture.

\begin{conjecture}\label{conjecture}
If the spectrum of a Noetherian ring has a Cohen-Macaulay blowing up, then the ring itself is catenary.
\end{conjecture}

The conjecture is stated for Noetherian rings which are not necessarily local. However it is enough to prove the conjecture for local rings. Indeed, if $\Proj \cR(R, I) \rightarrow \Spec(R)$ is a Macaulayfication for a Noetherian ring $R$ and $\fm$ is a maximal ideal of $R$, then the restriction morphism $\Proj \cR(R_\fm, IR_\fm) \rightarrow \Spec(R_\fm)$ is a Macaulayfication of $\Spec(R_\fm)$. Assuming the conjecture holds true for local rings, then $R_\fm$ is catenary. Hence $R$ is catenary. 

In order to reduce the original question to the case considered in Conjecture \ref{conjecture}, we need a result on Macaulayfication of polynomial rings.

\begin{proposition}\label{macaulayficationpolynomial}
Let $R$ be a Noetherian (not necessarily local) ring and $S:=R[T]$ be the polynomial ring in one indeterminate over $R$. Suppose there is an ideal $I\subset R$ such that the canonical morphism $\Proj \cR(R, I)\rightarrow \Spec(R)$ is a Macaulayfication of $\Spec(R)$. Then $\Proj \cR(S, IS)\rightarrow \Spec(S)$ is a Macaulayfication of the spectrum of the polynomial ring.
\end{proposition}
\begin{proof}
Let $P\in \Proj \cR(S, IS)$, and set $\frak p = P\cap \cR(R,I)$. Then we can check that  $\frak p \in \Proj \cR(R,I)$. In particular, $\cR(R, I)_\fp$ is Cohen-Macaulay from the assumption. Rewriting the graded ring $\cR(S, IS)$ by another way we have $\cR(S, IS) = \cR(R,I)[T]$. It implies $ \cR(S,IS)_P = (\cR(R, I)_\fp[T])_P$ which is Cohen-Macaulay. Therefore\linebreak $\Proj \cR(S, IS)$ is a Cohen-Macaulay scheme.
\end{proof}

This polynomial-base change theorem for Macaulayfication applies and we get almost directly the following theorem.

\begin{theorem}\label{catenary}
Assuming the answer to Conjecture \ref{conjecture} is positive. If the spectrum of a Noetherian ring has a Cohen-Macaulay blowing up, then the ring is universally catenary.
\end{theorem}
\begin{proof}
Let $R$ be a Noetherian ring such that $\Spec(R)$ has Cohen-Macaulay blowing up. The spectrum of the polynomial ring $R[T]$ has a Cohen-Macaulay blowing up by Proposition \ref{macaulayficationpolynomial}. Consequently, $R[T]$ is catenary provided Conjecture \ref{conjecture} has a positive answer. Therefore, $R$ is universally catenary by \cite[Corollary 1 of Theorem 31.7, page 255]{ma}.
\end{proof}

The formulation of Conjecture \ref{conjecture} is motivated by the works of Grothendieck on connection between the problem of resolution of singularities and excellent rings. In fact, it was proved in \cite{gr} that if the spectrum of a Noetherian ring has a resolution of singularities then all the formal fibers of the ring are regular. It seems that Grothendieck expected the ring to be also universally catenary, so excellent. Theorem \ref{catenary} could be seen as an attempt to prove this, under weaker assumption of existence of Macaulayfication.

\bigskip
\noindent{\bf Acknowledgments.} The authors thank M. Brodmann and S. Goto for  useful discussions and comments during this work. The second author thanks Vietnam Institute for Advanced Study in Mathematics (VIASM), Vietnam, and Institute for Mathematical Sciences (IMS-NUS), Singapore, for support and hospitality during his visit to these instituitions.


\end{document}